\documentclass[a4paper,11 pt]{amsart}

\usepackage{amssymb}  
\usepackage{mathtools}      
\usepackage{mathabx}        
\usepackage[bb=fourier,cal=euler,scr=rsfs]{mathalfa}	
\usepackage{enumitem}       
\usepackage[colorlinks=true,backref=page]{hyperref} 

\usepackage[ansinew]{inputenc}

\renewcommand*{\backref}[1]{}
\renewcommand*{\backrefalt}[4]{\quad \tiny
  \ifcase #1 (\textbf{NOT CITED.})%
  \or    (Cited on page~#2.)%
  \else   (Cited on pages~#2.)%
  \fi}

\hypersetup{bookmarksdepth = 3} 
\numberwithin{equation}{section}     

\setlist[enumerate,1]{label={\upshape(\roman*)},ref=\roman*}
\setlist[enumerate,2]{label={\upshape(\alph*)},ref=\alph*}

 \def\RR{{\mathbb R}}  
   
 \def\ZZ{{\mathbb Z}}

    \def\cW{\mathcal{W}}
\def\cF{\mathcal{F}}  \def\cL{\mathcal{L}}

\newtheorem*{teo*}{Theorem}

\newtheorem{teo}{Theorem}[section]

\newtheorem{lema}[teo]{Lemma}
\newtheorem{prop}[teo]{Proposition}
\newtheorem{add}[teo]{Addendum}

\newtheorem{thmintro}{Theorem}

\theoremstyle{definition}

\theoremstyle{remark}
\newtheorem{obs}[teo]{Remark}

\newcommand{\mt}{\widetilde M}
\newcommand{\ft}{\widetilde f}

\title[Partially hyperbolic diffeomorphisms homotopic to identity]{Partially hyperbolic diffeomorphisms homotopic to the identity on 3-manifolds \\ {\tiny RESEARCH ANNOUNCEMENT}}
\author[T.~Barthelm\'e]{Thomas Barthelm\'e}
\address{Queen's University, Kingston, ON}
\email{thomas.barthelme@queensu.ca}
\urladdr{sites.google.com/site/thomasbarthelme}

\author[S.~Fenley]{Sergio R.\ Fenley} 
\address{Florida State University, Tallahassee, FL 32306 \and \newline \indent Princeton University, Princeton, NJ 08540, USA}
\email{fenley@math.fsu.edu}

\author[S.~Frankel]{Steven Frankel} 
\address{Washington University in St Louis, St Louis, MO}
\email{steven.frankel@wustl.edu}

\author[R.~Potrie]{Rafael Potrie} 
\address{Centro de Matem\'atica, Universidad de la Rep\'ublica, Uruguay}
\email{rpotrie@cmat.edu.uy}
\urladdr{http://www.cmat.edu.uy/~rpotrie/}

\thanks{ R.~Potrie was partially supported by CSIC group 618. He also acknowledges the hospitality of Yale University and Laboratoire Math.~d'Orsay (funded by CNRS and IFUM) during part of the preparation of this work. S.~Fenley was partially supported by a grant from the Simons foundation. S.~Frankel supported by the National Science Foundation under Grant No. DMS-1611768. Any opinions, findings, and conclusions or recommendations expressed in this material are those of the authors and do not necessarily reflect the views of the National Science Foundation.}

\begin{document}

\begin{abstract}
We announce some results towards the classification of partially hyperbolic diffeomorphisms on $3$-manifolds, and outline the proofs in the case when the diffeomorphism is dynamically coherent. Detailed proofs are long and technical and will appear later. 
%
%
%
%
\end{abstract}

\maketitle

\section{Introduction}

A diffeomorphism $f$ of a $3$-manifold $M$ is \emph{partially hyperbolic} if it preserves a splitting of the tangent bundle $TM$ into three $1$-dimensional sub-bundles
\[ TM = E^s \oplus E^c \oplus E^u, \]
where the \emph{stable bundle} $E^s$ is eventually contracted, the \emph{unstable bundle} $E^u$ is eventually expanded, and the \emph{central bundle} $E^c$ is distorted less than the stable and unstable bundles at each point. That is, one has
\begin{align*}
					&\|Df^n|_{E^s(x)}\| < 1,\\
					&\|Df^n|_{E^u(x)}\| > 1,\\
	\text{and }		\|Df^n|_{E^s(x)}\| < &\|Df^n|_{E^{c}(x)}\| < \|Df^n|_{E^u(x)}\|
\end{align*}
for some $n \geq 1$ and all $x \in M$. See \cite{BDV,CHHU,HP-survey,CP,Wilkinson} for broad introductions to the subject. 

From a dynamical perspective, the interest in partial hyperbolicity stems from its appearance as a generic consequence of certain dynamical conditions, such as stable ergodicity and robust transitivity (see \cite{BDV, Wilkinson}). For example, recall that a diffeomorphism is \emph{transitive} if it has a dense orbit, and \emph{robustly transitive} if this behavior persists under $C^1$-small deformations. By \cite{DPU}, any robustly transitive diffeomorphism on a $3$-manifold is ``weakly'' partially hyperbolic.

From a geometric perspective, one can think of partial hyperbolicity as a generalization of the discrete behavior of Anosov flows, which feature prominently in the theory of $3$-manifolds. Recall that a flow $\Phi$ on a $3$-manifold $M$ is \emph{Anosov} if it preserves a splitting of the unit tangent bundle $TM$ into three $1$-dimensional sub-bundles
\[ TM = E^s \oplus E^c \oplus E^u, \]
where $E^c = T\Phi$ is the tangent direction to the flow, $E^s$ is eventually exponentially contracted, and $E^u$ is eventually exponentially expanded. After flowing for a fixed time, an Anosov flow generates a partially hyperbolic diffeomorphism of a particularly simple type, where the stable and unstable bundles are contracted uniformly, and the central direction is left undistorted. More generally, one can construct partially hyperbolic diffeomorphisms of the form $f(x) = \Phi_{t(x)}(x)$ where $\Phi$ is an Anosov flow and $t$ is a real-valued function; the diffeomorphisms obtained in this way are called \emph{discretized Anosov flows}.

In this article we announce a series of results towards the classification of partially hyperbolic diffeomorphisms in dimension $3$. In particular, we show that many partially hyperbolic diffeomorphisms can be identified with discretized Anosov flows. This is motivated by Pujals' conjecture which asserted that every partially hyperbolic diffeomorphism is a deformation of either a discretized Anosov flow or a certain kind of algebraic example \cite{BW}.

\subsection{Homotopy, integrability, and conjugacy}
There are two important obstructions to identifying a partially hyperbolic diffeomorphism with a discretized Anosov flow. The first comes from the fact that the latter are homotopic to the identity, while the former may be homotopically nontrivial. Examples include Anosov diffeomorphisms on the $3$-torus with distinct eigenvalues, ``skew products,'' and the counterexamples to the Pujals' conjecture constructed in \cite{BPP,BGP,BZ,BGHP}.

The second major obstruction comes from the integrability of the bundles in a partially hyperbolic spliting. In the context of an Anosov flow $\Phi$, the stable and unstable bundles $E^s$ and $E^u$ integrate uniquely into a transverse pair of $1$-dimensional foliations, the \emph{stable} foliation $\cW^{s}$ and \emph{unstable} foliation $\cW^{u}$. In fact, even the \emph{weak stable} and \emph{weak unstable} bundles $E^c \oplus E^s$ and $E^c \oplus E^u$ integrate uniquely into a transverse pair of $\Phi$-invariant $2$-dimensional foliations, the \emph{weak stable} foliation $\cW^{cs}$ and \emph{weak unstable} foliation $\cW^{cu}$.


In the context of a partially hyperbolic diffeomorphism $f$, the stable and unstable bundles still integrate uniquely into a pair of foliations $\cW^{s}$ and $\cW^{u}$ (cf. \cite{HP-survey} and references therein). However, the \emph{central stable} and \emph{central unstable} bundles $E^c \oplus E^s$ and $E^c \oplus E^u$ may fail to be uniquely integrable. In fact, there are examples where it is impossible to find \emph{any} $f$-invariant $2$-dimensional foliation tangent to the central stable or central unstable bundle; see \cite{HHU-noncoherent} (in ${\bf T}^3$)
 and \cite{BGHP} (in $T^1 S, S$ a hyperbolic surface).

If one can find a pair of $f$-invariant foliations tangent to the central stable and central unstable bundles then $f$ is said to be \emph{dynamically coherent}. This condition is obviously necessary for $f$ to be conjugate to a discretized Anosov flow, so we must either restrict attention to dynamically coherent diffeomorphisms, or show that dynamical coherence follows from other considerations.

\subsection{Results}
Most of the existing progress towards classifying partially hyperbolic diffeomorphisms takes an outside-in approach, restricting attention to particular classes of manifolds, and comparing to an \emph{a priori} known model partially hyperbolic (see \cite{CHHU,HP-survey} for recent surveys). In particular, partially hyperbolic diffeomorphisms have been completely classified in manifolds with solvable or virtually solvable fundamental group \cite{HP-Nil,HP-Sol}.

Ours is an inside-out approach, using the theory of foliations to understand the way the local structure that defines partial hyperbolicity is pieced together into a global picture. We then relate the dynamics of these foliations to the large-scale structure of the ambient manifold. In particular, we make use of several tools that had been previously employed to study Anosov, pseudo-Anosov, and quasigeodesic flows (see e.g.~\cite{BarbotFeuill,Barbot,Calegari,TBF1,TBF2,BarthelmeFenley14,BarthelmeFenley17,FenleyAnosov,Fen1998,Fen2002,Fen2003,Frankel} ). An advantage of this method is that, since it does not rely on a model partially hyperbolic to compare to, we can consider any manifold, not just one where an Anosov flow is known to exist.

The following two theorems are the main consequences of our work, applied to two of the major classes of $3$-manifolds. Note that the classification of partially hyperbolic diffeomorphisms is always considered up to finite iterates, since one can easily build infinitely many different but uninteresting examples by taking finite covers.

\begin{thmintro}[Seifert manifolds] \label{thmA}
	Let $f\colon M \to M$ be a partially hyperbolic diffeomorphism on a closed Seifert fibered $3$-manifold. If $f$ is homotopic to the identity, then after passing to a finite cover, some iterate of $f$ is conjugate to a discretized Anosov flow. 
\end{thmintro}

Note that the preceding theorem does not assume dynamical coherence, nor does it use that Anosov flows on Seifert fibered $3$-manifolds are classified \cite{Ghys,Barbot}. A weaker version of this theorem was recently announced by Ures \cite{Ures}, with the additional assumption that $f$ is isotopic, through partially hyperbolic diffeomorphisms, to the time-$1$ map of an Anosov flow.

\begin{thmintro}[Hyperbolic manifolds] \label{thmB}
	Let $f\colon M \to M$ be a partially hyperbolic diffeomorphism on a closed hyperbolic $3$-manifold. If $f$ is dynamically coherent, then some iterate of $f$ is conjugate to a discretized topological Anosov flow.
\end{thmintro}

Recall that most closed irreducible $3$-manifolds are hyperbolic; this was shown by Thurston, and follows for a more general sense of ``most'' from Perelman's proof of the Geometrization Conjecture. Note that Theorem~\ref{thmB} does not need to assume that $f$ is homotopic to the identity, since any homeomorphism on a closed hyperbolic $3$-manifold has a finite power that is homotopic to the identity. There are also more general versions of this theorem that do not require dynamical coherence, but we shall not state them here as they are more technical.

In Theorem~\ref{thmB}, the discretized Anosov flow that we construct is modelled on a \emph{topological} Anosov flow. It is currently unknown whether these can be taken to be smooth, though it is conjectured in general that every topological Anosov flow is orbit equivalent to a smooth Anosov flow \cite{BW}. In Theorem~\ref{thmA}, we make use of the fact that this is true for flows on Seifert fibered manifolds \cite{Brunella}.

The remainder of this announcement is organized as follows: In \S2-3 we develop the general theory of partially hyperbolic diffeomorphisms homotopic to the identity. We then specialize to Seifert manifolds in \S4, and to hyperbolic manifolds in \S5. In \S6 we discuss further work and the limits of our arguments.

\section{General Outline} 
Theorems \ref{thmA} and \ref{thmB} are best understood when seeing how they follow from our general study of partially hyperbolic diffeomorphisms homotopic to the identity. Many parts of the study hold in any 3-manifold whose fundamental group is not (virtually) solvable (where complete classification results are already available, see \cite{BI,Parwani,HP-Nil,HP-Sol}). We describe next these intermediate results.

For clarity, we will mainly discuss the dynamically coherent versions of our arguments, Many of these extend to the general case using the \emph{branching foliations} introduced in \cite{BI}, though in some places (notably, parts of Theorem~\ref{thmB}) the use of dynamical coherence is crucial; this has to do with phenomena very similar to the one appearing in recently obtained non-dynamically coherent examples \cite{BGHP}. 

For the remainder of this article, we fix a closed Riemannian $3$-manifold $M$ and a partially hyperbolic diffeomorphism $f \colon M \to M$ homotopic to the identity. Unless otherwise stated, $f$ will always be dynamically coherent, with invariant foliations $\cW^{\sigma}$ tangent to $E^\sigma$ where $\sigma=s,c,u,cs,cu$. This lifts to a diffeomorphism $\widetilde{f}$ on the universal cover $\widetilde M$, which leaves invariant the lifted foliations $\widetilde \cW^{\sigma}$. 

Since $f$ is homotopic to the identity, we can assume that the lift $\ft$ has the following properties:
\begin{itemize}
\item[(L1)] $\ft$ is bounded distance from the identity: There exists $K>0$ such that $d(x, \ft(x))< K$ for every $x\in \mt$. 
\item [(L2)] $\ft$ commutes with every deck transformation (which we identify with $\pi_1(M) \subset \mathrm{Isom}(\mt)$). 
\end{itemize}

\begin{obs}
	Such a lift can always be obtained by lifting an homotopy from the identity to $f$. Notice however that the choice of $\ft$ might not be unique (this will be important for Seifert manifolds). Whenever we write $\ft$ we will be assuming that $\ft$ is a lift that verifies both properties. This lift is fixed throughout.
\end{obs}

In this announcement, we will further assume that the foliations $\cW^{cs}$ and $\cW^{cu}$ are \emph{$f$-minimal}. This means that $M$ is the only set that is closed, saturated by the foliation, $f$-invariant and non empty. The difference with the usual notion of minimality of foliations is that we require the set to be $f$-invariant. This hypothesis simplifies several arguments but is not needed in certain cases (for instance when the manifold is Seifert or hyperbolic, although it requires some additional non-trivial arguments). Notice that this hypothesis is always verified in many important, from a dynamical standpoint, cases e.g.~when $f$ is transitive or volume preserving. 

Our main goal is to show that every leaf of both foliations $\widetilde \cW^{cs}$ and $\widetilde \cW^{cu}$ are fixed by $\ft$ and the same holds for the connected components of their intersections (i.e.~the leaves of $\widetilde \cW^c$). Once this is obtained, it is not difficult to show that $f$ should be leaf conjugate to a (topological) Anosov flow (very similar arguments already appear in \cite{BW}). 

Notice that invariance of $\widetilde \cW^\sigma$ means that if $L$ is a leaf of $\widetilde \cW^\sigma$ then so is $\ft(L)$. Showing that leaves are fixed means that for every $L \in \widetilde \cW^\sigma$ one has that $\ft(L)=L$.

\subsection{Dichotomy for foliations}
A foliation $\cF$ on $M$ is said to be \emph{$\RR$-covered} if the leaf space of the lifted foliation $\mt/_{\widetilde \cF}$ in the universal cover is homeomorphic to $\RR$. In general if $\cF$ is Reebless (for example if it does not have compact leaves), then $\mt/\cF$ is a simply connected one dimensional manifold, but usually it is \emph{not} Hausdorff \cite{Nov,Bar98}.
The foliation $\cF$ is called \emph{uniform}
if every pair of leaves of $\widetilde \cF$ 
in $\mt$ is a bounded Hausdorff distance apart  \cite{Thurston,CalegariPA,Fen2002}. The bound obviously depends on the particular pair.

Assuming that the foliations $\cW^{cs}$ and $\cW^{cu}$ are 
$f$-minimal in $M$ we first show that: 

\begin{prop}\label{p.dichotomy}
	Either every leaf of $\widetilde \cW^{cs}$ is fixed by $\ft$ or the foliation $\cW^{cs}$ is $\RR$-covered and uniform, and $\ft$ acts as a translation on the leaf space of $\widetilde \cW^{cs}$. 
The same dichotomy holds for $\widetilde \cW^{cu}$. 
\end{prop}
\begin{proof}[Sketch]
Since the $2$-dimensional foliations do not have compact leaves, they are \emph{taut}. In particular both connected components of the complement of a leaf $L$ in $\mt$ contain arbitrarily large balls. This implies that the image of a leaf must be \emph{nested} with itself, i.e.~for a fixed transverse orientation, the positive half-space determined by one leaf contains the positive half-space determined by the other. This way, if a leaf $L$ is not fixed by $f$, one can consider the set $V \in \mt$ defined by

$$ V:=\bigcup_n \ft^n(L \cup U),$$
\noindent where $U$ is the region `between' $L$ and $\tilde f(L)$. It follows that the set $V$ can be shown to be open and $\tilde f$-invariant. 
Using that $\ft$ commutes with deck transformations and that the image of a leaf is nested with itself one can show that the boundary leaves of $V$ are also invariant under $\ft$ and therefore the set $V$ verifies that for every deck transformation $\gamma \in \pi_1(M)$ one has that either $\gamma V = V$ or $\gamma V \cap V = \emptyset$. By $f$-minimality we obtain that $V = \mt$ and this implies the second possibility. Additional work is needed to show that 
$\cW^{cs}$ is $\RR$-covered.
\end{proof}

Given this proposition there are three possibilities:
\begin{enumerate}[label=\arabic*)]
 \item $\ft$ fixes every leaf of both $\widetilde \cW^{cs}$
and $\widetilde \cW^{cu}$, referred to as the \emph{doubly invariant} case;
 \item $\ft$ fixes no leaves of either foliation, henceforth called the
\emph{double translation} case; and
 \item $\ft$ fixes every leaf of one of the foliations, but
no leaf of the other foliation, henceforth called the \emph{mixed} case.
\end{enumerate}
Our goal is to rule out 2) and 3).

\subsection{No mixed behavior} 

We can show (this will be expanded upon later): 

\begin{prop}\label{p.nomix} 
If $M$ is hyperbolic or Seifert and $\ft$ fixes a leaf of 
$\widetilde \cW^{cs}$ then it fixes every leaf of 
\emph{both} $\widetilde \cW^{cs}$ and 
$\widetilde \cW^{cu}$. 
\end{prop} 

%



\subsection{Double translation} 
In order to be leaf conjugate to the time one map of a (topological) Anosov flow one needs to exclude the possibility that either
of the foliations $\widetilde \cW^{cs}$ or 
$\widetilde \cW^{cu}$ are translated by $\ft$. 

The proof of this is very different in the Seifert and the hyperbolic case (and we do not know how to make it work for more general manifolds).
In the hyperbolic case it depends crucially on dynamical coherence. 
In the Seifert case we can (after considering a finite iterate) choose a different lift such that all the leaves of one of foliations in $\mt$ are fixed by that new lift $\ft$. 
This is enough to exclude this possibility (for this specific lift). 

In the hyperbolic manifold case the proof is much more involved and 
uses the existence of a transverse pseudo-Anosov flow to the
$\RR$-covered foliation (either $\cW^{cs}$ or 
$\cW^{cu}$). This forces a particular
 dynamics on periodic center leaves. Using both translations it is possible to find a contradiction (see Proposition \ref{prop-translation-expanding} and Proposition \ref{prop-contracting-nonDC}). 

\subsection{Double invariance}\label{ss.doublyinvariant}

Once we know that both foliations are fixed by $\ft$, 
the next step is to show that connected components of the intersections between leaves of $\widetilde \cW^{cs}$ and $\widetilde \cW^{cu}$ (i.e.~center leaves $-$ $\widetilde \cW^c$) are also fixed by 
$\ft$. In turn, after some more or less standard considerations (see also \cite{BW}), this yields the desired statement 
that $f$ is leaf conjugate to the time one map of a topological Anosov flow in Theorems \ref{thmA} and \ref{thmB}.

The key point of this stage is to show that the set of fixed leaves
of $\widetilde \cW^c$ is open and closed in $\mt$.
From this, if the set of fixed centers was to be empty, we can apply Proposition \ref{prop.Key} below to obtain a contradiction. Showing that the set of 
fixed center leaves is open is not so complicated, but closedness is a bit more delicate.

\subsection{Important property}
As will be explained later, the proof that the mixed
case or the double translation case cannot happen under 
certain situations, is achieved as follows:
We analyze the structure
forced by the hypothesis of mixed or double translation
situation and we prove that in a center leaf that is
periodic under $f$, we have that both rays have to
be (say) contracting, and at the same time one of the
rays has to be expanding. So the analysis of the action
of $f$ on rays of periodic center leaves is crucial to
our strategy.

\section{A key general proposition} 

In this section we analyze the case where one assumes
that one of the foliations, (say) $\widetilde \cW^{cs}$ is 
leafwise fixed by $\ft$. 
A symmetric analysis holds for $\widetilde \cW^{cu}$.

\subsection{Consequences of fixed center-stable leaves}

The first relatively simple but powerful consequence of having 
$cs$-leaves fixed was already noted in \cite{BW} (see also \cite{HP-Sol}): 

\begin{lema}\label{l.nofix}
The lift $\ft$ has no fixed (or periodic) points. 
\end{lema}

This is fairly simple. 
Suppose that $x$ in  a leaf $L$ of 
$\widetilde \cW^{cs}$ is fixed by $\ft$ and consider the
unstable leaf $u(x)$ of $x$.
The intersection of an unstable (one dimensional) leaf in $\mt$
with a center stable (two dimensional) leaf is at most a single
point \cite{Nov}.
Since both $u(x)$ and any $L'$ center stable leaf are fixed
by $\ft$, then every point in $u(x)$ is fixed by $\ft$.
This contradicts the fact that iteration by $\ft$ pushes points
in $u(x)$ `away' from $x$.
It follows that no fixed or periodic points of $\ft$ can exist. 

Using this, and the fact that $\ft$ contracts the one dimensional
stable leaves, one deduces that the action of 
$\ft$ on the space of stable leaves 
\emph{in a fixed leaf of $\widetilde \cW^{cs}$} is free (i.e.~it has no fixed points). Similarly, since the 
stable foliation (in $M$) is by lines (it contains no circles) one also knows that the action of every deck transformation in the space of stable leaves is also free. Putting this together with the fact that $\ft$ commutes with deck transformation and using the theory of axes for actions on non-Hausdorff, simply connected one manifolds
 (see e.g.~\cite{Bar98, Fen98, Fen2003});
 one deduces the following very important property: 

\begin{prop}\label{p.cylorplane} 
Every leaf of $\cW^{cs}$ is a cylinder, a plane, or a M\"{o}bius band.
\end{prop}

Note also that by a result of Rosenberg \cite{Ros} not every leaf (in $M$) can be a plane. Hence, by $f$-minimality, we get that cylinder and M\"{o}bius leaves are dense in $M$.

\subsection{Gromov hyperbolic leaves}\label{ss.Gromov}

For foliations by surfaces on 3-manifolds one has the following important result: 

\begin{teo}[\cite{Sul,Gro}]\label{t.Gromov}
Let $\cF$ be a codimension one foliation with no compact
leaves on a closed 3-manifold $M$. Then, either there exists a transverse invariant measure, or the leaves of $\cF$ are Gromov hyperbolic.
\end{teo}


Sullivan \cite{Sul} proved that leaves satisfy a linear isoperimetric
inequality. Later, Gromov \cite[section 6.8]{Gro} proved that it implies Gromov hyperbolicity of the leaves. 
This result also follows from Candel's uniformization 
theorem \cite{Candel}.

In our setting (either $M$ is hyperbolic or Seifert or the foliations are 
$f$-minimal), using partial hyperbolicity, we can show that the foliation cannot admit a transverse invariant measure (in the case of all $\widetilde \cW^{cs}$ leaves fixed by $\ft$). 

\subsection{Coarse contraction and a key proposition} 
We start by defining a property 
that will be helpful in our study. We say that a center leaf $c \in \cW^c$ is \emph{coarsely contracting} if:

\begin{itemize}
\item it is homeomorphic to $\RR$ (i.e.~not a circle),
\item it is periodic by $f$ (i.e.~there exists $k$ such that $f^k(c)=c$),
\item it has a bounded interval $I$ containing all the fixed points of $f^k$,
\item every point $y \in c \smallsetminus I$ converges to $I$ under forward iteration of $f^k$. 
\end{itemize}

We say that a center leaf is \emph{coarsely expanding} if $f^{-1}$ is coarsely contracting. 
Here is the first result concerning dynamics on periodic center leaves.
In this result we do not assume dynamical coherence.

\begin{prop}\label{prop.Key}
Let $f\colon M \to M$ be a partially hyperbolic diffeomorphism (not
necessarily dynamically coherent) such that 
$\ft$ fixes every leaf of $\widetilde \cW^{cs}$ and does not fix any leaf of 
$\widetilde \cW^c$. Assume moreover that the foliations $\cW^{cs}$ and $\cW^{cu}$ are $f$-minimal. Then, every periodic center leaf $c$ of $f$ is coarsely contracting. Moreover, there is at least one coarsely contracting periodic center leaf. 
\end{prop}

To prove this Proposition we use some properties of deck 
transformations of $\widetilde M$ fixing a leaf of 
$\widetilde \cW^{cs}$. By Theorem \ref{t.Gromov}, these
leaves are Gromov hyperbolic, hence isometries restricted
to the leaves are hyperbolic \cite{Gro}. This hyperbolic behavior will be crucial in our analysis.

\begin{add}\label{add-key}
Suppose the hypothesis of Proposition \ref{prop.Key} are
satisfied. Suppose moreover that $f$ is dynamically
coherent. Then the dynamical behavior 
described in Proposition \ref{prop.Key} is impossible.
In other words, if $\ft$ fixes every leaf of $\widetilde \cW^{cs}$ then 
it has to fix a leaf of $\widetilde \cW^c$.
\end{add}

So far, we have to assume dynamical coherence to get the Addendum in this generality. We can easily prove this result without the assumption for Seifert manifolds, and, with a lot more work, for hyperbolic manifolds. It is not yet clear to us whether the assumption is really needed in the general case.



The proof of the proposition above 
is quite involved but we can sum up the main idea as follows:

\begin{proof}[Sketch of the proof]
Up to a finite cover and iterates we may assume that there
are no M\"{o}bius band leaves in $\cW^{cs}$.
Since $\tilde f$ has no periodic points, it follows that every periodic point of $f$ has to be in a cylinder leaf. 

Take a cylinder leaf and its lift $L$ 
to $\mt$ which is stabilized by a deck transformation $\gamma$. 
Since the action of $\ft$ is free on the stable foliation in $L$,
there is an axis for the action on the stable 
leaf space in $L$. The first thing to notice is that a graph transform argument shows that a center leaf in $L$ cannot intersect a leaf $s$ of $\cW^s$ and a translate $\gamma^k s$ for some $k \neq 0$ as that would produce 
a fixed center leaf for $\ft$ contradicting the hypothesis. 

Since every leaf of $\widetilde \cW^{cs}$ is fixed by $\ft$ 
one can show that $\ft$ is a bounded distance from the identity \emph{in the induced metric on $L$}. This and the previous remark allows to obtain a structure on the leaf $L$ where, essentially, the leaf $L$ is covered by bands of bounded width between a stable leaf $s$ and its translate by $\gamma$. 
Notice that since $L/\gamma$ is an annulus with a hyperbolic metric,
then width of $L/\gamma$ itself goes to infinity as one
escapes into the ends of $L/\gamma$.
Moreover, every center leaf gets trapped in such a `band' (i.e.~the translates $\gamma^k s$ bound bands which fill up $L$). 
There is an iterate $\ell>0$ and $k \in \ZZ\setminus\{0\}$ such that $h = \gamma^k \circ f^\ell$ fixes each band (in particular, fixes $s$). Notice that $h$ is still a partially hyperbolic diffeomorphism and has a fixed point $x \in s$. 

From Candel's theorem, we can assume that $\gamma$ acts on $L$ as a hyperbolic isometry. Thus, we can show that there are points in $s$ which are mapped by $h$ arbitrarily far (i.e.~for every $R>0$ there are points $z \in s$ in both sides of $s \setminus \{x\}$ such that $d_L(h(z),z) > R$). This in turn yields that all points in between $s$ and $\gamma s$ are mapped in a trapping way and provides the coarse contraction on centers. 

When $f$ is dynamically coherent, this behavior is impossible, and this provides the addendum (the behavior is very similar to that of the examples in \cite{HHU-noncoherent} and a similar argument shows that this cannot happen).
\end{proof} 

\section{Seifert Manifolds}\label{s.Seifert}
In this section, let $f\colon M \to M$ be a dynamically coherent partially hyperbolic diffeomorphism with a lift $\ft$ as described before.

We denote by $\mathfrak{c}$ a generator of the center of $\pi_1(M)$ which corresponds to the fibers of the Seifert fibration.  

\subsection{Horizontality} 
It was shown in \cite{HaPS} that in the setting of this section one has that both $\cW^{cs}$ and $\cW^{cu}$ are \emph{horizontal} (i.e.~leaves are 
uniformly transverse to the Seifert fibers after isotopy). This is relevant to show that both foliations must be minimal (and therefore one can apply 
Theorem \ref{t.Gromov}).


\subsection{Changing the lifts} 

Since the fundamental group of a Seifert manifold have a non-trivial center, a trick that we can use to simplify a lot our analysis is to chose our lift $\ft$ well:
\begin{prop}\label{p.changelift}
Let $\ft\colon \mt \to \mt$ be a lift of $f$ at bounded distance from the identity and commuting with deck transformations. Then, there exists $\ell >0, k \in \ZZ$ 
such that $\mathfrak{c}^k \circ \ft^\ell$ is a lift of $f^\ell$ which is at bounded distance from the identity, commutes with deck transformations, and fixes a leaf of $\widetilde \cW^{cs}$. 
\end{prop}

\subsection{Putting information together}\label{ss.seif}
Using Proposition \ref{p.changelift} we can choose an iterate $f^k$ of $f$ which admits two lifts $\ft_1$ and $\ft_2$ one of which 
fixes all leaves of $\widetilde \cW^{cs}$ and the other fixes 
every leaf of $\widetilde \cW^{cu}$. (Notice that we could apply directly Addendum \ref{add-key} but we rather explain this slightly longer argument that is generalizable to the non dynamically coherent setting.)

Assuming that the lifts do not coincide (i.e.~we are in the `mixed behavior case') one gets a contradiction by applying Proposition \ref{prop.Key} to both lifts. Indeed, the proposition implies that all periodic center leaves must be both coarsely contracted and coarsely expanded by $f^k$, and since the proposition also ensures the existence of at least one periodic center leaf, we get a contradiction. This gives a lift $\ft$ of $f^k$ which leafwise fixes every leaf of both $\widetilde \cW^{cs}$ and $\widetilde \cW^{cu}$. 

Once this is obtained, we argue as in subsection \ref{ss.doublyinvariant}: it is possible to show that the set of fixed center leaves is either everything or empty, and in the latter case one can again apply Proposition \ref{prop.Key} to both foliations to get a contradiction. This shows that every center leaf is 
fixed by $\ft$ and this is enough to complete the proof of Theorem \ref{thmA}.

\section{Hyperbolic Manifolds} 
In this section we explain the main tools that need to be added to work out the case of hyperbolic 3-manifolds and $f$ dynamically
coherent  (Theorem \ref{thmB}). Some of the arguments can be carried out in more generality (e.g.~without assuming dynamical coherence) but others use dynamical coherence in a crucial way as we will explain below.

\subsection{Uniform foliations and transverse pseudo-Anosov flows}
Following \cite{Thurston} (see also \cite{Calegari,Fen2002}) we say that a foliation $\cF$ of a 3-manifold $M$ is $\RR$-covered and
\emph{uniform} if the following two properties hold: 
\begin{itemize}
\item The leaf space $\cL:= \mt/_{\widetilde \cF}$ is homeomorphic to $\mathbb{R}$ and,
\item for every pair of leaves $L,L' \in \widetilde \cF$, there exists $K>0$ such that the Hausdorff distance between $L$ and $L'$ is less than $K$. 
\end{itemize}

When $M$ is obtained as the suspension of a pseudo-Anosov diffeomorphism of a surface $S$ (i.e.~$M = S \times [0,1]/_{(x,0) \sim (\varphi(x),1)}$) it is clear that the foliation by fibers $S \times \{t\}$ is
$\RR$-covered and  uniform, and admits a transverse pseudo-Anosov flow. This is an instance of a much more general result dealing with general uniform foliations in hyperbolic 3-manifolds:   

\begin{teo}[Thurston, Calegari, Fenley \cite{Thurston, CalegariPA,Fen2002}]\label{t.transversePA} 
If $\cF$ is a transversely orientable, 
$\RR$-covered and  uniform foliation in a hyperbolic 3-manifold $M$ then it admits a \emph{regulating} transverse pseudo-Anosov flow $\Phi$. 
\end{teo}

By \emph{regulating} we mean that every orbit of $\widetilde \Phi$ in $\mt$ 
intersects every leaf of $\widetilde \cF$. Being transverse just says that orbits of $\Phi$ are transverse to $\cF$. In our proof, we use this result in an essential way to eliminate the double translation case in hyperbolic 3-manifolds. We do it by comparing the dynamics $f$ with that of the pseudo-Anosov flow $\Phi$.

\subsection{Forcing a particular type of dynamics on periodic center leaves} 

\begin{prop}\label{prop-translation-expanding} 
Let $f\colon M \to M$ be a dynamically coherent partially hyperbolic diffeomorphism of a hyperbolic manifold $M$ such that $\ft$ acts as a translation on $\widetilde \cW^{cs}$ then, there is a periodic center leaf which is coarsely expanding. 
\end{prop}


\begin{proof}[Sketch of the proof of Proposition \ref{prop-translation-expanding}]
Recall that, according to our dichotomy result (Proposition \ref{p.dichotomy}), if $\ft$ acts as a translation on $\widetilde \cW^{cs}$ then $\cW^{cs}$ is $\RR$-covered and uniform. Let $\Phi_{cs}$ be the regulating transverse pseudo-Anosov flow given by Theorem \ref{t.transversePA}. 

Consider $\gamma$ a periodic orbit of $\Phi_{cs}$ and write $\gamma$ again for an associated deck transformation.

The first step in the proof consists in showing that there exist $\ell>0$ and $k \in \ZZ \setminus \{0\}$ such that $h = \gamma^k \circ \ft^\ell$ fixes a leaf 
$L \in \widetilde \cW^{cs}$. This is shown using that $\ft$ is a \emph{bounded distance from $\widetilde \Phi_{cs}$}, understood in the following way: Flowing along $\widetilde \Phi_{cs}$ defines a homeomorphism between the leaf
$L$ and $\ft(L)$ and this homeomorphism is a bounded distance
from the map $\ft|_L$ from $L$ to $\ft(L)$.
After some involved arguments we obtain a compact 
$\ft/\gamma$ invariant subset in 
$\mt/_{\gamma}$. Then, using recurrence and partial hyperbolicity, we get the desired periodic center stable leaf. 

Once this is obtained, we use Lefschetz fixed point theorem to compare the indices of fixed points of $h$ in $L$ with the corresponding first return map of the flow $\widetilde \Phi_{cs}$. This forces the existence of at least one fixed center leaf whose index is negative, which produces the desired coarsely expanding leaf (because the transverse behavior is contracting because $L$ is a center-stable leaf). 
\end{proof}

\subsection{Obstructions to dynamical coherence, no double translation}

\begin{prop}\label{prop-contracting-nonDC}
Let $f\colon M \to M$ be a dynamically coherent partially hyperbolic diffeomorphism of a hyperbolic $3$-manifold $M$ such that $\ft$ acts as a translation on 
$\widetilde \cW^{cs}$. Then $f$ cannot have any coarsely contracting center leaf.  
\end{prop}


If $f$ is dynamically coherent, then putting Proposition \ref{prop-translation-expanding} applied to 
$\widetilde \cW^{cu}$ together with Proposition \ref{prop-contracting-nonDC} 
applied to $\widetilde \cW^{cs}$ yields that the double translation case cannot happen.  Unfortunately, our proof of this result uses dynamical coherence in a crucial way, as we will see in the sketch below.

\begin{proof}[Sketch of the proof of Proposition \ref{prop-contracting-nonDC}]
Let $L$ be a center stable leaf of
$\widetilde \cW^{cs}$
fixed by $h = \gamma \circ \ft^\ell$ for some deck 
transformation $\gamma$ (in the terminology of the previous
proposition $\gamma = \gamma^k$). 
Assume that $L$ contains a coarsely contracting fixed center stable leaf $c$. 

This proof requires a finer study of the dynamics forced by 
$\widetilde \Phi_{cs}$ on 
$L$. We separate this study in two cases, determined by whether or not $\gamma$ corresponds to a periodic orbit of $\Phi_{cs}$. Both are very similar so we only sketch the case where $\gamma$ does correspond to a periodic orbit. 

In this case, we start by showing that the contracting rays of the center leaf $c$ must accumulate (in $\partial_{\infty}L$, the boundary at infinity of the leaf $L$) on points which are repelling for the action of $\tau_{cs}$ on the boundary, where $\tau_{cs}\colon L \to L$ is the map obtained by composing the holonomy along $\tilde \Phi_{cs}$-orbits from $L$ to $\gamma^{-1}L$ with $\gamma$.  Similarly, any stable manifold that is periodic under $h$ also accumulates only
on repelling points of $\tau_{cs}$ on $\partial_{\infty}L$. 
Notice that $h$ and $\tau_{cs}$ are homeomorphisms of $L$ a bounded
distance from each other, so induce the same action on $\partial_{\infty}L$.

An index counting argument shows that it is impossible to compensate the (positive) index contributed by the coarsely contracting center leaf with other coarsely expanding centers (because there are only finitely many contracting points at infinity) 
unless some center leaves merge, contradicting dynamical coherence. Notice that this type of merging for non-dynamically coherent diffeomorphisms actually appears in examples, as in \cite[Section 5]{BGHP}.
\end{proof}

\subsection{No mixed behavior} 
 The impossibility of having mixed behavior is proven, for dynamically coherent diffeomorphisms, by Addendum \ref{add-key}. 
As we previously mentioned, we are also able to eliminate mixed behavior on hyperbolic manifolds even in the non-dynamically coherent case.

However, our argument is very specific to hyperbolic manifolds (since we use the existence of a transverse pseudo-Anosov flow) and considerably more delicate than the dynamically coherent case. At best, this argument could be extended to manifolds with at least one atoroidal piece (see below). 


\subsection{Doubly invariant case} 

This follows exactly as in the Seifert case (cf.~subsection \ref{ss.seif}). 

\section{Extensions and limits of our arguments}

\subsection{Beyond dynamical coherence}
When $f$ is not necessarily dynamically coherent, one can use, instead of foliations, the branching foliations introduced in \cite{BI} (see also \cite[Section 4]{HP-survey}). After a substantial amount of preparation, 
and suitable reinterpretation of objects (like leaf spaces),
most of the arguments of the dynamically coherent situation extend
to the non dynamically coherent setting.
Some properties require different and involved arguments, notably showing that in the hyperbolic and Seifert case the branching foliations are $f$-minimal, and we sometimes need additional hypothesis. One step that we are so far unable to complete is to remove dynamical coherence from the assumptions in Proposition \ref{prop-contracting-nonDC} --- that is when 
$M$ is hyperbolic. Indeed, the type of configuration that we obtain using our arguments turns out to be very similar to what actually happens in the non-dynamically coherent examples constructed in \cite{BGHP} in some Seifert manifolds.
Therefore, it is unclear whether this situation in hyperbolic
manifolds can really be ruled out.

Notice that once we can prove double invariance, then
dynamical coherence follows after the fact: Once we have shown that all branching leaves, as well as the connected components of their intersections are fixed, we can deduce that the branching foliations are true foliations, i.e.~the partially hyperbolic diffeomorphism is dynamically coherent.
Since all our arguments can be extended to the non dynamically coherent case when $M$ is 
Seifert fibered, we obtain Theorem \ref{thmA}.


\subsection{Double translation in hyperbolic manifolds}

One particularly obvious gap so far is our inability to either prove or disprove the existence of a \emph{double translation} example in hyperbolic manifolds (such examples would necessarily be non dynamically coherent).

We previously explained why our method has failed so far, but to understand the intricacy of this problem, the reader can meditate on the following example: Let $\phi^t$ be a
(smooth) Anosov flow on a hyperbolic manifold $M$, such that its (say) weak stable ($2$-dimensional) foliation is $\RR$-covered (such flows are called $\RR$-covered Anosov flows).
Since the flow is $\RR$-covered, there exists a map $\eta\colon M \rightarrow M$ that conjugates $\phi^t$ with its inverse $\phi^{-t}$ (see, for instance \cite[Proposition 2.4]{BarthelmeFenley17} for a description of $\eta$). In particular, $\eta$ preserves the weak stable and weak unstable foliations of $\phi^t$, and acts as a translation on both leaf spaces. 

If $\eta$ was $C^1$, then it would be easy to show that, for a time $T_0$ big enough, the map $f = \phi^{T_0}\circ \eta^2$ would be a partially hyperbolic diffeomorphism. Moreover, one can easily deduce from known facts about Anosov flows in dimension $3$ that $f$ could not be leaf conjugate to a time-$1$ map of an Anosov flow. This is essentially
because the only Anosov flows  that can be transverse to a $\RR$-covered foliation are orbit equivalent to suspensions. 
Since $f$ preserves the weak stable and weak unstable foliations of $\phi^t$, these foliations must be the center stable and center unstable foliations. In particular, $f$ would be dynamically coherent and act as a double translation, in contradiction with Theorem \ref{thmB}. 

It follows that $\eta$ cannot be $C^1$ (actually, Barbot \cite[Proposition 6.6]{BarbotHDR} proved that, in general, the map $\eta$ is $C^1$ if and only if the flow $\phi^t$ is a lift of a geodesic flow).

A natural question is: \emph{Does there exists a $C^1$-map $h$, $C^0$-close to $\eta$, and such that $h$ sends the strong unstable (resp.~stable) leaves of $\phi^t$ to curves transverse to the weak stable (resp.~unstable) foliation of $\phi^t$?} If the answer is yes, then, for big enough $T_0$, the map $\phi^{T_0}\circ h$ will be a partially hyperbolic diffeomorphism (see e.g.~\cite[Section 2]{BGHP}), homotopic to the identity, not the time-1 map of an Anosov flow, and acting as a double translation (hence not dynamically coherent).

\subsection{More general manifolds}

Seifert and hyperbolic 3-manifolds are a large part of 
the family of irreducible 3-manifolds (which are the only ones that can 
admit partially hyperbolic diffeomorphisms \cite{Nov,BI}). 
However, the case of general irreducible 3-manifolds, even under the assumption of being homotopic to identity still requires further work (though as we have mentioned, our results also provide some progress in this general case). 

 Several arguments, starting with the dichotomy, require $f$-minimality of the (branching) foliations (which can be obtained in the contexts of Theorems \ref{thmA} and \ref{thmB}). Even assuming minimality some arguments do not carry directly in general. Notably we do not know how to rule out the double translation case in general. We remark, though, that it is reasonable to expect an analogue of Theorem \ref{t.transversePA} in the context of manifolds whose JSJ decomposition contains at least one atoroidal piece. Indeed, it is believed (see, for instance, \cite[Remark 5.3.17]{CalegariPA}) that a regulating flow that behaves like a pseudo-Anosov inside the atoroidal piece do exist, extending Theorem \ref{t.transversePA} to that setting. 
This, together with additional work, might be enough
to extend Propositions \ref{prop-translation-expanding} and \ref{prop-contracting-nonDC} in this setting.
The remaining case is when $M$ does not have atoroidal pieces, but $M$ is not Seifert, that is, when $M$ is a graph manifold. In that case there is no ``pseudo-Anosov"-like
flow in any piece, nor does the fundamental 
group admit a center, which prevents us to use our Seifert trick. Hence to analyze this case, one will need new ideas.

%

\end{document}